\documentclass[12pt]{amsart}

\usepackage{amsmath,amssymb,amsbsy,amsfonts,amsthm,latexsym,
            amsopn,amstext,amsxtra,amscd,stmaryrd,color,bm}
                                               
\usepackage{float}
\usepackage[colorlinks,linkcolor=blue,anchorcolor=blue,citecolor=blue]{hyperref}

\hypersetup{breaklinks=true}

\usepackage[mathscr]{eucal}

\begin{document}

\newtheorem{theorem}{Theorem}
\newtheorem{lemma}[theorem]{Lemma}
\newtheorem{corollary}[theorem]{Corollary}
\newtheorem{proposition}[theorem]{Proposition}

\theoremstyle{definition}
\newtheorem*{definition}{Definition}
\newtheorem{remark}[theorem]{Remark}
\newtheorem*{example}{Example}

\numberwithin{table}{section}
\numberwithin{equation}{section}
\numberwithin{figure}{section}
\numberwithin{theorem}{section}


\def\cA{\mathcal A}
\def\cB{\mathcal B}
\def\cC{\mathcal C}
\def\cD{\mathcal D}
\def\cE{\mathcal E}
\def\cF{\mathcal F}
\def\cG{\mathcal G}
\def\cH{\mathcal H}
\def\cI{\mathcal I}
\def\cJ{\mathcal J}
\def\cK{\mathcal K}
\def\cL{\mathcal L}
\def\cM{\mathcal M}
\def\cN{\mathcal N}
\def\cO{\mathcal O}
\def\cP{\mathcal P}
\def\cQ{\mathcal Q}
\def\cR{\mathcal R}
\def\cS{\mathcal S}
\def\cU{\mathcal U}
\def\cT{\mathcal T}
\def\cV{\mathcal V}
\def\cW{\mathcal W}
\def\cX{\mathcal X}
\def\cY{\mathcal Y}
\def\cZ{\mathcal Z}


\def\sA{\mathscr A}
\def\sB{\mathscr B}
\def\sC{\mathscr C}
\def\sD{\mathscr D}
\def\sE{\mathscr E}
\def\sF{\mathscr F}
\def\sG{\mathscr G}
\def\sH{\mathscr H}
\def\sI{\mathscr I}
\def\sJ{\mathscr J}
\def\sK{\mathscr K}
\def\sL{\mathscr L}
\def\sM{\mathscr M}
\def\sN{\mathscr N}
\def\sO{\mathscr O}
\def\sP{\mathscr P}
\def\sQ{\mathscr Q}
\def\sR{\mathscr R}
\def\sS{\mathscr S}
\def\sU{\mathscr U}
\def\sT{\mathscr T}
\def\sV{\mathscr V}
\def\sW{\mathscr W}
\def\sX{\mathscr X}
\def\sY{\mathscr Y}
\def\sZ{\mathscr Z}


\def\fA{\mathfrak A}
\def\fB{\mathfrak B}
\def\fC{\mathfrak C}
\def\fD{\mathfrak D}
\def\fE{\mathfrak E}
\def\fF{\mathfrak F}
\def\fG{\mathfrak G}
\def\fH{\mathfrak H}
\def\fI{\mathfrak I}
\def\fJ{\mathfrak J}
\def\fK{\mathfrak K}
\def\fL{\mathfrak L}
\def\fM{\mathfrak M}
\def\fN{\mathfrak N}
\def\fO{\mathfrak O}
\def\fP{\mathfrak P}
\def\fQ{\mathfrak Q}
\def\fR{\mathfrak R}
\def\fS{\mathfrak S}
\def\fU{\mathfrak U}
\def\fT{\mathfrak T}
\def\fV{\mathfrak V}
\def\fW{\mathfrak W}
\def\fX{\mathfrak X}
\def\fY{\mathfrak Y}
\def\fZ{\mathfrak Z}


\def\C{{\mathbb C}}
\def\F{{\mathbb F}}
\def\K{{\mathbb K}}
\def\L{{\mathbb L}}
\def\N{{\mathbb N}}
\def\P{{\mathbb P}}
\def\Pc{\P^{\hskip1pt c}}
\def\Q{{\mathbb Q}}
\def\R{{\mathbb R}}
\def\Z{{\mathbb Z}}


\def\eps{\varepsilon}
\def\mand{\qquad\mbox{and}\qquad}
\def\mor{\qquad\mbox{or}\qquad}
\def\\{\cr}
\def\({\left(}
\def\){\right)}
\def\[{\left[}
\def\]{\right]}
\def\<{\langle}
\def\>{\rangle}
\def\fl#1{\left\lfloor#1\right\rfloor}
\def\rf#1{\left\lceil#1\right\rceil}
\def\le{\leqslant}
\def\ge{\geqslant}
\def\ds{\displaystyle}
\def\lealmost{\preccurlyeq}
\def\gealmost{\succcurlyeq}

\def\xxx{\vskip5pt\hrule\vskip5pt}
\def\yyy{\vskip5pt\hrule\vskip2pt\hrule\vskip5pt}
\def\imhere{ \xxx\centerline{\sc I'm here}\xxx }
\def\gettohere{ \xxx\centerline{\sc gotta get to here...}\xxx }

\newcommand{\commB}[1]{\marginpar{%
\begin{color}{red}
\vskip-\baselineskip 
\raggedright\footnotesize
\itshape\hrule \smallskip B: #1\par\smallskip\hrule\end{color}}}

\newcommand{\commI}[1]{\marginpar{%
\begin{color}{blue}
\vskip-\baselineskip 
\raggedright\footnotesize
\itshape\hrule \smallskip I: #1\par\smallskip\hrule\end{color}}}

\newcommand{\commG}[1]{\marginpar{%
\begin{color}{cyan}
\vskip-\baselineskip 
\raggedright\footnotesize
\itshape\hrule \smallskip G: #1\par\smallskip\hrule\end{color}}}


\def\e{\mathbf{e}}
\def\sPrc{{\displaystyle \sP_r^{(c)}}}


\title[Arithmetic properties of numbers of the form $\fl{p^c}$]
{\bf Some arithmetic properties of\\ numbers of the form $\fl{p^c}$}

\author[W.~D.~Banks]{William D.~Banks}  
\address{Department of Mathematics, University of Missouri, 
Columbia, MO 65211 USA}
\email{bankswd@missouri.edu} 

\author[V.~Z.~Guo]{Victor Z.~Guo} 
\address{Department of Mathematics, University of Missouri, 
Columbia, MO 65211 USA}
\email{zgbmf@mail.missouri.edu}

\author[I.~E.~Shparlinski]{Igor E.~Shparlinski}  
\address{Department of Pure Mathematics, University of New South Wales, 
Sydney, NSW 2052, Australia}
\email{igor.shparlinski@unsw.edu.au} 

\date{}
\pagenumbering{arabic}

\begin{abstract}
Let
$$
\Pc=(\fl{p^c})_{p\in\P}\qquad (c>1,~c\not\in\N),
$$
where $\P$ is the set of prime numbers, and $\fl{\cdot}$ is the floor function.
We show that for every such $c$ there are infinitely many members of $\Pc$ having
at most $R(c)$ prime factors, giving explicit estimates for $R(c)$ when $c$ is near
one and also when $c$ is large.
\end{abstract}

\maketitle

\section{Introduction}

\subsection{Motivation}
\emph{Piatetski-Shapiro sequences} are those sequences of the form
$$
\N^c=(\fl{n^c})_{n\in\N}\qquad (c>1,~c\not\in\N),
$$
where $\fl{t}$ denotes the integer part of any real number $t$.
Such sequences are named in honor of Piatetski-Shapiro, who showed
(cf.~\cite{PS})
that for any fixed $c\in(1,\tfrac{12}{11})$ there
are infinitely many primes in $\N^c$.
The admissible range of $c$ for this result has been
extended many times over the years, and currently it is
known to hold for all $c\in(1,\tfrac{243}{205})$ thanks
to Rivat and Wu~\cite{RivatWu}.

Many authors have studied arithmetic properties of Piatetski-Shapiro sequences
(see Baker \emph{et al}~\cite{BakerPSseqs} and the references contained
therein), and it is natural to ask whether certain properties also hold on
special subsequences of the Piatetski-Shapiro sequences.  Perhaps the most
important of these are the subsequences of the form
$$
\Pc=(\fl{p^c})_{p\in\P}\qquad (c>1,~c\not\in\N),
$$
where $\P=\{2,3,5,\ldots\}$ is the set of prime numbers; however,
up to now very little has been established about the arithmetic
structure of $\Pc$ for fixed~$c>1$.  Balog~\cite{Balog} has shown that
for \emph{almost all} $c>1$, the counting function
$$
\Pi_c(x)=\bigl|\bigl\{\text{prime~}p\le x:\fl{p^c}\text{~is prime}\bigr\}\bigr|
$$
satisfies
$$
\limsup_{x\to\infty}\frac{\Pi_c(x)}{x/(c\log^2x)}\ge 1,
$$
but this result gives no information for any specific choice of $c$.

Thanks to the work of Cao and Zhai~\cite{CaoZhai} it is known that the set $\Pc$
contains infinitely many \emph{squarefree} natural numbers provided that $c$ is not
too large.  More precisely, as a special case of the main result in \cite{CaoZhai},
one knows that for any $c\in(1,\frac{149}{87})$ there exists $\eps>0$
(depending only on $c$) such that the estimate
$$
\bigl|\bigl\{\text{\rm prime~}p\le x
:\fl{p^c}\text{\rm ~is squarefree}\bigr\}\bigr|
=\frac{6}{\pi^2}\cdot\pi(x)+O(x^{1-\eps})
$$
holds, where $\pi(x)$ denotes the number of primes not exceeding $x$.

In the present paper, as a step towards better understanding the arithmetic
properties of $\Pc$, we consider the related question of whether or not $\Pc$ contains
infinitely many \emph{almost primes}.

\subsection{Main results}

For every $R\ge 1$, we say that a natural number is an
\emph{$R$-almost prime} if it has at most $R$ prime factors,
counted with multiplicity.

We study almost prime values of $\fl{p^c}$
in two different regimes in order to demonstrate the underlying ideas:
(i) values of $c$ close to one, and
(ii) large values of $c$.

In the first regime, our result is stated in terms of 
the following set of admissible pairs $(R, c_R)$, $R =8, \ldots, 19$.
\begin{table}[H]
\centering
\begin{tabular}{|c|c||c|c||c|c|}
\hline
\vphantom{\Big|}$R$  &  $c_R$ & $R$  &  $c_R$ & $R$  &  $c_R$ \\
\hline
$8$ & $1.0521$ & $12$ & $1.1649$ & $16$ & $1.2073$ \\
$9$ & $1.1056$ & $13$ & $1.1780$ & $17$ & $1.2148$ \\
$10$ & $1.1308$ & $14$ & $1.1891$ & $18$ & $1.2214$ \\
$11$ & $1.1494$ & $15$ & $1.1988$ & $19$ & $1.2273$ \\
\hline
\end{tabular}
\caption{Admissible pairs $(R, c_R)$}
\label{tab:c_R}
\end{table}

\begin{theorem}
\label{thm:main2}
Let $(R,c_R)$,  $R =8, \ldots, 19$, be a pair from Table~\ref{tab:c_R}. 
Then for any fixed $c\in(1,c_R]$ there is a real number $\eta>0$ such that the lower bound 
$$
\bigl|\bigl\{\text{\rm prime~}p\le x
:\fl{p^c}\text{\rm ~is an $R$-almost prime}\bigr\}\bigr|
\ge \eta\,\frac{x}{\log^2x}
$$
holds for all sufficiently large $x$.
\end{theorem}

In the second regime, we prove the following result.

\begin{theorem}
\label{thm:main3}
For fixed $c\ge\frac{11}5$ there is a positive integer
$$R \le 
\begin{cases}  
16c^3+179c^2&\quad\hbox{if $c\in[\frac{11}5,3)$},\\
16c^3+88c^2&\quad\hbox{if $c\ge 3$},
\end{cases}
$$
and a real number $\eta>0$
such that the lower bound
$$
\bigl|\bigl\{\text{\rm prime~}p\le x
:\fl{p^c}\text{\rm ~is an $R$-almost prime}\bigr\}\bigr|
\ge \eta\,\frac{x}{\log^2x}
$$
holds for all sufficiently large $x$.
\end{theorem}

These results are based on bounds of bilinear exponential sums 
and estimates on the uniformity of distribution 
of fractional parts $\{p^cd^{-1}\}$.
We use the notion of \emph{level of distribution} from
sieve theory in a precise form stated in~\S\ref{sec:level distr};
see Friedlander and Iwaniec~\cite{FrIw} and Greaves~\cite{Greaves}
We remark that although the ranges of Theorems~\ref{thm:main2}
and~\ref{thm:main3} do not overlap, using the same methods and
sacrificing on the explicitness of the bounds for $R$,
one can cover the gap as well.

\subsection{Notation}
\label{sec:note}

Throughout the paper, we use the symbols $O$, $\ll$, $\gg$ and $\asymp$
along with their standard meanings; any constants or functions
implied by these symbols may depend
on $c$ and (where obvious) on the parameters $\eps$
and $\nu$ but are absolute otherwise.
We use the notation $m\sim M$ as an
abbreviation for $M<m\le 2M$.

The letter $p$ always denotes a prime number.
As usual, $\mu(\cdot)$ is the
M\"obius function, and $\Lambda(\cdot)$ is
the von Mangoldt function.

We write $\e(t)=\exp(2\pi it)$ for all $t\in\R$.

\section{Proof of Theorem~\ref{thm:main2}}

\subsection{Preliminaries}
\label{sec:level distr} 

As we have mentioned the following notion plays a
crucial r\^ole in our arguments. We specify it to the form 
that is suited to our applications; it is based 
on a result of Greaves~\cite{Greaves}
that relates level of distribution to $R$-almost primality.
More precisely, we say that an $N$-element set of integers $\cA$
has a \emph{level of distribution} $D$ if for a given multiplicative
function $f(d)$ we have 
$$
\sum_{d \le D} \max_{\gcd(s,d)=1} \left|\big|\{a\in \cA,~
a \equiv s \bmod d\}\big|
- \frac{f(d)}{d}N\right| \le  \frac{N}{\log^2N}.
$$
As in~\cite[pp.~174--175]{Greaves} we define the constants
$$
\delta_2 = 0.044560, \qquad \delta_3 = 0.074267, 
\qquad \delta_4 = 0.103974
$$
and 
$$
\delta_R = 0.124820, \qquad R \ge 5.
$$
We have the following result, which is~\cite[Chapter~5, Proposition~1]{Greaves}.

\begin{lemma}
\label{lem:Greaves}
Suppose $\cA$ is an $N$-element set of positive integers 
with a level of distribution $D$ and degree $\rho$ in the sense
that
$$
a<D^\rho \qquad(a\in \cA)
$$
holds with some real number $\rho<R-\delta_R$. Then 
$$
\big|\{a \in \cA:\text{\rm $a$~is an $R$-almost prime}\}\big|
\gg_\rho  \frac{N}{\log^2N}\,.
$$
\end{lemma}

Note that we always have $R\ge 5$ in what follows.

Using Baker and Pollack~\cite[Lemma~1]{BaPol} together with Lemma~\ref{lem:Greaves},
it is easily seen that the proof of Theorem~\ref{thm:main2} reduces to
showing that, for a fixed pair $(R,c_R)$ as in Table~\ref{tab:c_R}, 
for any fixed numbers $c\in(1,c_R]$ and $\vartheta\in(0,1/R)$
the uniform bound
\begin{equation}
\label{eq:triple sum}
\sum_{1\le h\le H}\sum_{d \sim D}  \left| \sum_{n\sim x}
\Lambda(n) \e(h d^{-1} n^c) \right|
\ll_{\vartheta} \frac{Dx}{\log^3x}
\end{equation}
holds with any $D\le x^\vartheta$ and $H=D\log^3x$.
To estimate the triple sums in~\eqref{eq:triple sum} we
treat the summation over $h$ with straightforward estimates
after estimating the inner sums over $d$ and $n$.
Choosing a sufficiently small $\kappa>0$ and applying
Rivat and Sargos~\cite[Lemma~2]{RivSar} with
$$
\alpha = \max\{1/20, \vartheta + \kappa\} < 1/6,
$$  
it suffices to show that 
\begin{equation}
\label{eq:multilinear_sum_hyperb}
\sum_{d\sim D}c_d\sum_{m\sim M}a_m\sum_{\ell\sim x/m}b_\ell\,\e(h d^{-1} \ell^c m^c) 
\ll_{\vartheta, \kappa, \xi}  x^{1-\xi}
\end{equation}
with some fixed $\xi > 0$ (depending on $\vartheta$),
arbitrary weights $c_d, a_m, b_\ell$ of size $O(1)$,
and in three ranges of  $M$
that correspond to two Type~I sums and one Type~II sum. 
More precisely, denoting
$$
u_0 = x^\alpha
$$
these ranges are the following:
\begin{itemize}
\item[(i)] Type~II sums: $u_0 \ll x/M \ll u_0^2$;
\item[(ii)]  Type~I sums: $u_0^2  \ll x/M \ll x^{1/3}$
with $b_\ell$ being the characteristic function of an interval;
\item[(iii)]  Type~I sums: $M \ll x^{1/2} u_0^{1/2}$
with $b_\ell$ being the characteristic function of an interval. 
\end{itemize}
By a standard application of the Fourier analysis (see, e.g., Garaev~\cite{Gar} or
Banks \emph{et al}~\cite{BFGS})
the hyperbolic region of summation in~\eqref{eq:multilinear_sum_hyperb} can be
replaced with a rectangular region; in other words,
it is enough to derive the bound 
\begin{equation}
\label{eq:multilinear_sum}
\sum_{d\sim D}c_d\sum_{m\sim M}a_m\sum_{\ell\sim L}b_\ell\,\e(h d^{-1} \ell^c m^c) 
\ll_{\vartheta, \kappa, \xi}  x^{1-\xi}
\end{equation}
for some $L$ and $M$ with $LM \asymp x$ in the following  three ranges:
\begin{itemize}
\item[(i)] Multilinear  Type~II sums: $u_0 \ll L \ll u_0^2$; 
\item[(ii)]  Multilinear   Type~I sums: $u_0^2  \ll L \ll x^{1/3}$
with $b_\ell$ being the characteristic function of an interval;
\item[(iii)]   Multilinear   Type~I sums: $M \ll x^{1/2} u_0^{1/2}$
with $b_\ell$ being the characteristic function of an interval.
\end{itemize}

Before proceeding, we record the following technical result
which simplifies the exposition below.

\begin{lemma}
\label{lem:useful}
Fix an admissible pair $(R,c_R)$ from Table~\ref{tab:c_R}.
For any fixed numbers $c\in(1,c_R]$ and $\vartheta\in(0,1/R)$,
there is a positive number $\kappa$ such that if we define
$$
\alpha = \max\{1/20, \vartheta + \kappa\},
$$
then all of the following inequalities hold:
\begin{itemize}
\item[(i)] $2\vartheta+2\alpha<c$;
\item[(ii)] $c+5\vartheta+2\alpha<2$;
\item[(iii)] $365/3+32c+147\vartheta<174$;
\item[(iv)] $8/3+c+2\vartheta<4$;
\item[(v)] $2+c+4\vartheta<4$;
\item[(vi)] $1+\vartheta-2\alpha<1$;
\item[(vii)] $1+\vartheta/2-\alpha<1$;
\item[(viii)] $2/3+\vartheta<1$;
\item[(ix)] $1-c/2+3\vartheta/2<1$;
\item[(x)] $2\vartheta+(1+\alpha)/2<c$;
\item[(xi)] $2c+6\vartheta+\alpha<3$.
\end{itemize}
\end{lemma}

\begin{remark}
These inequalities are listed for convenience only and in some cases are
redundant (for instance, (vi) and (vii) are equivalent).
The proof of Lemma~\ref{lem:useful} is straightforward.
\end{remark}

\subsection{General multilinear  sums} 

First, we need an adaptation of a result of Baker~\cite[Theorem~2]{Baker1}, 
which is given here only for the specific exponent pair
$(\kappa, \lambda) = (\frac12,\frac12)$. 
Note that we use $D$ and $L$
instead of $M_1$ and $M_2$, respectively in the notation of~\cite[Theorem~2]{Baker1},
and thus we use $d$ and $\ell$ instead of $m_1$ and $m_2$.
However, $M$ and $m$ retain the same meaning.   

\begin{lemma}
\label{lem:GenMultSum} 
Let  $\alpha_1,\alpha_2,\beta$ be nonzero real numbers such that $\beta<1$, let $h,D,L,M$
be positive integers, and let $g$ be a real function on the interval $[M,2M]$
such that
$$
g'(x)\asymp h M^{\beta-j}\qquad (x\sim M).
$$
Let
$$
S=\sum_{m \sim M}\sum_{d \sim D} \sum_{\ell \sim L} 
a_m  c_{d, \ell}\,
\e\(g(m) d^{\alpha_1}  \ell^{\alpha_2}\) 
$$
where $a_m,c_{d,\ell}$ are complex numbers with
$a_m,c_{d,\ell}\ll 1$. If the number
$X = h D^{\alpha_1} L^{\alpha_2} M^\beta$ is such that $X\ge DL$,
then
$$
S \ll DLM\((DL)^{-1/2} + (X/(DLM^2))^{1/6}\)\log 2DL.
$$
\end{lemma}

\begin{proof}
As this is a straightforward variant of~\cite[Theorem~2]{Baker1} we
indicate mainly the changes that are needed in the proof.

Let
\begin{equation}
\label{eq:Q small}
 Q \le DL
\end{equation}
be a natural number to be determined later.
Following \cite{Baker1} we see that either
(cf.~\cite[Equation~(3.8)]{Baker1})
\begin{equation}
\label{eq:skate0}
S^2\ll DLM^2Q\sL^2
\end{equation}
holds with $\sL=\log 2DL$ (which corresponds to the value $h=0$ 
in \cite[Equation~(3.6)]{Baker1}),
or else we have (cf.~\cite[Equation~(3.9)]{Baker1})
\begin{equation}
\label{eq:skate1}
S^2\ll D^2L^2MQ\Delta\sL^2\bigg|\sum_{m\sim M}\e(f(m))\bigg|,
\end{equation}
where $f(x)=g(x)(d_1^{\alpha_1}\ell_1^{\alpha_2}-d_2^{\alpha_1}\ell_2^{\alpha_2})$
with some quadruple $(d_1,d_2,\ell_1,\ell_2)$ that satisfies
$$
d_1,d_2\sim D,\qquad
\ell_1,\ell_2\sim L,\qquad
\Delta-\frac{1}{DL}\le\biggl|\Bigl(\frac{d_1}{d_2}\Bigr)^{\alpha_1}
-\Bigl(\frac{\ell_2}{\ell_1}\Bigr)^{\alpha_2}\biggr|<2\Delta
$$
where $\Delta$ is a number of  the form 
$\Delta = 2^h(DL)^{-1}$ with 
some fixed integer $h\ge 1$, which satisfies the bound
\begin{equation}
\label{eq:Delta small}
\Delta\ll Q^{-1}
\end{equation}
(recall also the condition~\eqref{eq:Q small}).
Note that
$$
f'(m)\asymp X\Delta M^{-1}\qquad (x\sim M)
$$
as in~\cite{Baker1}.

Now, if the inequality  $X\Delta M^{-1}\le \varepsilon$ holds
with for some sufficiently small (but fixed) $\varepsilon >0$,
we can proceed as in Case~(i) in the 
proof of \cite[Theorem~2]{Baker1} (making use of~\cite[Lemma~4.19]{Titch}) 
to obtain the bound
$$
\sum_{m\sim M} \e(f(m)) \ll X^{-1}\Delta^{-1}M.
$$
Since $X\ge DL$, upon combining this with~\eqref{eq:skate1} we again
obtain~\eqref{eq:skate0}.

On the other hand, if the inequality $X\Delta M^{-1}>\varepsilon$ holds,
then we can proceed as in Case~(ii) in the proof of~\cite[Theorem~2]{Baker1}
(with $\kappa=\lambda=\tfrac12$) to derive that
$$
\sum_{m\sim M} \e(f(m)) \ll   ( X\Delta)^{1/2}.
$$
Combining this with~\eqref{eq:skate1} and~\eqref{eq:Delta small} we have
\begin{equation}
\label{eq:skate4}
S^2\ll D^2L^2M\sL^2(X/Q)^{1/2}.
\end{equation}

Putting~\eqref{eq:skate0} and~\eqref{eq:skate4} together, we deduce that
$$
S\ll DLM\sL\((Q/(DL))^{1/2}+(X/(M^2Q))^{1/4}\).
$$
The optimal choice for the natural number $Q$ is
$$
Q=\rf{(D^2L^2X/M^2)^{1/3}}.
$$
We note that if for  the above choice of $Q$ condition~\eqref{eq:Q small}
is not satisfied  then $X/M^2 \gg DL$ 
and the  result is trivial.  Now, simple calculations lead to the desired 
bound. 
\end{proof} 

\subsection{Multilinear  sums: Region$\,$(i)}
In this region, we can apply Lemma~\ref{lem:GenMultSum} 
to bound the sum in~\eqref{eq:multilinear_sum},
making the choices $\alpha_1=-1$, $\alpha_2=\beta=c$, 
$c_{d, \ell} = c_d b_\ell$ and $g(x) = hx^c$. Since
$LM\asymp x$ and $2\vartheta+2\alpha<c$ by Lemma~\ref{lem:useful}$\,$(i)
we see that
\begin{equation}
\label{eq:library1}
X = h D^{-1} L^c M^c\ge DL
\end{equation}
if $x$ is large, and recalling that $H=D\sL^3$ with $\sL=\log x$
we also have 
$$
X \ll H  D^{-1} x^c = x^c\sL^3;
$$
hence, for the sum
$$
S=\sum_{d\sim D}c_d\sum_{m\sim M}a_m\sum_{\ell\sim L}b_\ell\,\e(h d^{-1} \ell^c m^c)
$$
Lemma~\ref{lem:GenMultSum} yields
\begin{align*}
S&\ll DLM\sL\((DL)^{-1/2} + (x^c\sL^3/(DL))^{1/6} M^{-1/3}\)\\
&\ll x\((D/L)^{1/2}\sL + (D^5x^c/(LM^2))^{1/6} \sL^{3/2}\).
\end{align*}
In Region$\,$(i) we have
$LM^2\gg x^2/L \gg x^2u_0^{-2}$, and therefore
\begin{equation}
\label{eq:Reg i}
S\ll x \((D/L)^{1/2}\sL + (D^5x^{c-2} u_0^2)^{1/6} \sL^{3/2}\).
\end{equation}
Recalling our choice of $u_0$, in Region$\,$(i) we have 
$$
L \ge u_0 \ge x^{\vartheta + \kappa} \ge Dx^{\kappa}; 
$$
hence the first term in~\eqref{eq:Reg i} is of size $O(x^{1-\kappa/2})$.
For the second term in~\eqref{eq:Reg i}, Lemma~\ref{lem:useful}$\,$(ii) implies that
the inequality
$$
5\vartheta+(c-2)+2\alpha<-\kappa
$$
holds with a suitably small $\kappa$, hence the second term in~\eqref{eq:Reg i} is
of size $O(x^{1-\kappa/2})$ as well.

\subsection{Multilinear  sums: Region$\,$(ii)} 
In this region, to estimate the sum
$$
S=\sum_{d\sim D}c_d\sum_{m\sim M}a_m\sum_{\ell\sim L}b_\ell\,\e(h d^{-1} \ell^c m^c)
$$
in~\eqref{eq:multilinear_sum} we apply a result of Wu~\cite{Wu}.
Note that $b_\ell$ is a characteristic function of an
interval. The correspondence between the 
parameters $(H,M,N,X,\alpha,\beta,\gamma)$ given in~\cite[Theorem~2]{Wu} and 
our parameters is 
$$
(H,M,N,X,\alpha,\beta,\gamma) \quad \longleftrightarrow \quad (M,D,L,X,c,-1,c)
$$
(where $X = h D^{-1} L^c M^c$ as before)
and we take $k=5$ in the statement of~\cite[Theorem~2]{Wu}; this gives
\begin{align*}
S\sL^{-1}
&\ll \bigl(X^{32}M^{114}D^{147}L^{137}\bigr)^{1/174}
+(XM^2D^2L^4)^{1/4}
+(XM^2D^4L^2)^{1/4}\\
&\qquad\qquad\qquad\quad+MD+M(DL)^{1/2}
+M^{1/2}DL+X^{-1/2}MDL.
\end{align*}
Using the bounds
$$
D^{-1}x^c\le X\ll x^c\sL^3,\quad
LM\asymp x,\quad
x^{2\alpha}\ll L\ll x^{1/3}\quad\text{and}\quad
D\le x^\vartheta,
$$
it follows that
\begin{align*}
S\sL^{-2}
&\ll \bigl(x^{365/3+32c+147\vartheta}\bigr)^{1/174}
+(x^{8/3+c+2\vartheta})^{1/4}
+(x^{2+c+4\vartheta})^{1/4}\\
&\qquad \qquad\qquad+x^{1+\vartheta-2\alpha}+x^{1+\vartheta/2-\alpha}
+x^{2/3+\vartheta}+x^{1-c/2+3\vartheta/2}.
\end{align*}
Taking into account the inequalities of Lemma~\ref{lem:useful}$\,$(iii)--(ix)
we see that $S=O(x^{1-\kappa})$ if $\kappa>0$ is small enough.

\subsection{Multilinear  sums: Region$\,$(iii)} 

In this region, to estimate the sums in~\eqref{eq:multilinear_sum}
we apply a result of Robert and Sargos~\cite{RobSar}.
Note that $b_\ell$ is a characteristic function of an
interval. The correspondence between the 
parameters $(H,M,N,X,\alpha,\beta,\gamma)$ given in~\cite[Theorem~3]{RobSar} and 
our parameters is 
$$
(H,M,N,X,\alpha,\beta,\gamma) \quad \longleftrightarrow \quad (D,L,M,X,c,-1,c), 
$$
where 
$$
X=h D^{-1}L^c M^c.
$$ 
Applying~\cite[Theorem~3]{RobSar}, for the sum
$$
S=\sum_{d\sim D}c_d\sum_{m\sim M}a_m\sum_{\ell\sim L}b_\ell\,\e(h d^{-1}\ell^c m^c)
$$
we have the bound
$$
S\le (DLM)^{1+ o(1)}\(\(\frac{X}{DL^2M}\)^{1/4}  + \frac{1}{L^{1/2}}
+ \frac{1}{X}\).
$$
The third term in this estimate is dominated by the second term since
$X\ge DL$ (cf.~\eqref{eq:library1}), and the second term is dominated
by the first term since $X\ge DM$, the latter bound holding in Region~(iii)
in view of the inequality $c\ge 2\vartheta+(1+\alpha)/2$
in Lemma~\ref{lem:useful}$\,$(x). Therefore, 
$$
S\le (DLM)^{1+ o(1)}\(\frac{h D^{-1}L^c M^c}{DL^2M}\)^{1/4}.
$$
Since $h \le H=D \sL^3$, $LM\asymp x$, $D\le x^\vartheta$
and $L\gg x^{1/2}u_0^{-1/2}$, we have
$$
S\le x^{5/8+c/4+3\vartheta/4+\alpha/8+o(1)}.
$$
To prove~\eqref{eq:multilinear_sum} in this case
it is enough to show that
$$
5/8+c/4+3\vartheta/4+\alpha/8<1,
$$
This follows from the inequality
$$
2c+6\vartheta+\alpha<3, 
$$
which is given in Lemma~\ref{lem:useful}$\,$(xi).

\section{Proof of Theorem~\ref{thm:main3}}

\subsection{Preliminaries}
 
Let $c$ be fixed, and put
\begin{equation}
\label{eq:sigma_defn}
\sigma=\frac{1}{16c^2+179c-1.15c^{-1}}
\mand
\beta=47\sigma.
\end{equation}
For our purposes below, we record that the inequality
\begin{equation}
\label{eq:c1ineq}
\frac{c_1(\tfrac12-\beta)^3-(\tfrac12-\beta)^4}
{(c_1+\tfrac12-\beta)(c_1+1-2\beta)(2c_1+\tfrac12-\beta)}>\sigma
\end{equation}
holds with $c_1=c+\sigma$ for all $c\ge 2.081$, and the inequalities
\begin{equation}
\label{eq:c2ineqA}
\frac{\tfrac8{27}c_2-\tfrac{16}{81}}{(c_2+\tfrac43)(c_2+2)(2c_2+2)}>2\sigma
\end{equation}
and
\begin{equation}
\label{eq:c2ineqB}
\frac{c_2(1-2\beta)^3-(1-2\beta)^4}{(c_2+2-4\beta)(c_2+3-6\beta)(2c_2+3-6\beta)}>2\sigma
\end{equation}
both hold with $c_2=c-1+3\sigma$ for all $c\ge 2.198$.

Suppose that we have the uniform bound
\begin{equation}
\label{eq:dig}
\sum_{p\le x}\e(hd^{-1}p^c)\ll x^{1-\sigma}\qquad(d,h\le x^\sigma).
\end{equation}
Let $\sA$ be the sieving set given by
$$
\sA=\bigl\{n:n=\fl{p^c}\text{~for some prime $p\le x$}\bigr\},
$$
If~\eqref{eq:dig} holds, then (as in the proof of Theorem~\ref{thm:main2})
for any fixed $\eps>0$ we obtain a level of
distribution $D=x^{\sigma-\eps}$ for $\sA$.  Thus,
we can apply Lemma~\ref{lem:Greaves}
with $g=c/\sigma+\eps$ (since $a\le x^c$ for all $a\in\sA$) and with
\begin{equation}
\label{eq:cheeto}
R\le\frac{c}{\sigma}+1.15=16c^3+179c^2,
\end{equation}
which implies the stated result for $c\in[\frac{11}{5},3)$.

For $c\ge 3$ we replace  $179$ with $88$ in the 
definition of $\sigma$ and take $\beta=20\sigma$ in~\eqref{eq:sigma_defn},
and the estimates~\eqref{eq:c2ineqB}--\eqref{eq:cheeto} continue to hold
(as well as the bound $\beta<0.1$; see \S\ref{sec:concluding} below).
Hence, we can also
replace $179$ with $88$ in~\eqref{eq:cheeto} as well. 

\subsection{Bounds on some auxiliary sums}
\label{sec:aux sum} 

Here, it is convenient to introduce the notations $A\lealmost B$ 
and $B \gealmost A$, which are equivalents of an inequality of
the form $A\le B+O(\sL^{-1})$, where $\sL=\log N$. 

To prove that~\eqref{eq:dig} holds,
we need the following bound of exponential sums; it is used 
to establish~\eqref{eq:solar} and~\eqref{eq:panther} below.

\begin{lemma}
\label{lem:gerbils}
Let $c,\Theta,\Delta,\eps>0$ be fixed, and put
\begin{equation}
\label{eq:k_defn}
k=\fl{c+\Delta/\Theta}+1.
\end{equation}
If $k\ge 3$, then the exponential sum
$$
S(N)=\sum_{z\sim N^\Theta}\e(z^cN^\Delta)
$$
satisfies the bound
\begin{equation}
\label{eq:S(N)bd}
S(N)\ll N^{\Theta(1-\varrho)},
\end{equation}
where the implied constant depends only on $c$ and $\eps$, and
\begin{equation}
\label{eq:rho_defn}
\varrho=\frac{k-2-\eps}{k(k+1)(2k-1)}.
\end{equation}
\end{lemma}

\begin{proof}
Let $s=k^2-1$.
Applying the result of Vinogradov~\cite[Chapter~VI, Lemma~7]{Vino}
with the function $F(z)=z^cN^\Delta$
and $n=k$, for any fixed $\varrho\in(0,1)$ we have the bound
\begin{equation}
\label{eq:Toronto}
S(N)^{2s}\ll P^{-2s+\frac12k(k+1)}(N^\Theta)^{2s-1+2/k+(k+1)\varrho}~\cI
+(N^\Theta)^{2s(1-\varrho)},
\end{equation}
where
$$
\cI=\int_0^1\!\!\!\cdots\!\int_0^1
\bigl|\sum_{z=1}^{P}\e(\alpha_1z+\cdots+\alpha_k z^k)\bigr|^{2s}
\,d\alpha_1\ldots d\alpha_k
$$
and $P$ is the integer given by
$$
P=\fl{A_0^{(1-\varrho)/(k+1)}},\qquad\text{where}\quad
A_0=\bigg|\frac{(k+1)!}{F^{(k+1)}(N^\Theta)}\bigg|.
$$
Noting that
$$
A_0\asymp N^{\Theta(k+1-c)-\Delta},
$$
in order to apply~\cite[Chapter~VI, Lemma~7]{Vino} it must be the case that
$$
\Theta\lealmost\Theta(k+1-c)-\Delta\lealmost\Theta(2+2/k),
$$
or in other words,
$$
c+\Delta/\Theta\lealmost k\lealmost 1+2/k+c+\Delta/\Theta.
$$
However, this condition is guaranteed by~\eqref{eq:k_defn}.

Applying Wooley~\cite[Theorem~1.1]{Wooley} with $\eps/k$ in place of $\eps$,
we see that the integral $\cI$ is bounded by
\begin{equation}
\label{eq:Toronto2}
\cI\ll P^{2s-\frac12k(k+1)+\eps/k}.
\end{equation}
Taking into account that
$$
P\ll A_0^{1/(k+1)}\ll N^{(\Theta(k+1-c)-\Delta)/(k+1)}\le N^\Theta,
$$
after combining~\eqref{eq:Toronto} and~\eqref{eq:Toronto2}
we derive the bound
$$
S(N)^{2s}\ll (N^\Theta)^{2s-1+(2+\eps)/k+(k+1)\varrho}~\cI
+(N^\Theta)^{2s(1-\varrho)}.
$$
To optimize, we choose $\varrho$ so that
$$
2s-1+(2+\eps)/k+(k+1)\varrho=2s(1-\varrho);
$$
recalling that $s=k^2-1$ this leads to~\eqref{eq:rho_defn},
and~\eqref{eq:S(N)bd} follows.
\end{proof}

\subsection{Concluding the proof}
\label{sec:concluding}
We now turn our attention to~\eqref{eq:dig}.
We use the Heath-Brown decomposition (cf.~Heath-Brown~\cite{HB}) to reduce
the problem to that of bounding Type~I and Type~II
sums. In the present situation,
to prove~\eqref{eq:dig}
it suffices to show, for some sufficiently small $\eps>0$ which depends only on~$c$,
that $B=N^{1-\sigma-\eps}$ is an upper bound on all Type~I sums
\begin{equation}
\label{eq:S1defx}
S_I(X,Y)=\sum_{x\sim X}\sum_{y\sim Y}
a_x\,\e(hd^{-1}x^cy^c)\qquad(Y\gg N^{\frac12-\beta})
\end{equation}
and an upper bound on all Type~II sums
\begin{equation}
\label{eq:S2defx}
S_{I\!I}(X,Y)=\sum_{x\sim X}\sum_{y\sim Y}
a_x\,b_y\,\e(hd^{-1}x^cy^c)\qquad(N^{2\beta}\ll Y\ll N^{\frac13}),
\end{equation}
where $|a_x|\le 1$ and $|b_y|\le 1$,
and $XY\asymp N$; we refer the reader to
the discussion on~\cite[pp.~1367-1368]{HB}.
We specify $\eps>0$ below.

Let $\sL=\log N$ as before.
Using van der Corput's inequality with
$Q=N^{2\sigma+2\eps}$  and
following the proof of Baker~\cite[Theorem~5]{Baker2},
we are lead to the bound~\cite[Equation~(4.18)]{Baker2}
with some $q\in[1,Q]$:
\begin{align*}
S_{I\!I}(X,Y)^2\sL^{-2}
&\ll \frac{N^2}{Q}+\frac{N\sL q}{Q}
\biggl|\sum_{y\sim Y}\sum_{x\sim X}b_{y+q}\overline{b_y}
\,e\bigl(hd^{-1}x^c((y+q)^c-y^c)\bigr)\biggl|\\
&\ll N^{2-2\sigma-2\eps}+N\sL\sum_{y\sim Y}
\biggl|\sum_{x\sim X}
e\bigl(hd^{-1}x^c((y+q)^c-y^c)\bigr)\biggl|.
\end{align*}
For the moment, put $\Theta=(\log X)/\sL$, so that $X=N^\Theta$.
Noting that
$$
qY^{c-1}N^{-\sigma}\ll hd^{-1}((y+q)^c-y^c)\ll qY^{c-1}N^\sigma\qquad(y\sim Y),
$$
and taking into account that
$$
Y\asymp N^{1-\Theta}\mand 1\le q\le N^{2\sigma+2\eps},
$$
we see that in the Type~II case it suffices to show that
\begin{equation}
\label{eq:solar}
\sum_{z\sim N^\Theta}\e(z^cN^\Delta)\ll N^{\Theta-2\sigma-3\eps}
\end{equation}
holds uniformly for
\begin{equation}
\label{eq:solar2}
2/3\lealmost\Theta\lealmost 1-2\beta
\end{equation}
and
\begin{equation}
\label{eq:solar3}
(1-\Theta)(c-1)-\sigma\lealmost\Delta\lealmost (1-\Theta)(c-1)+3\sigma+2\eps,
\end{equation}
where continue to use the notation $A\lealmost B$ 
from \S\ref{sec:aux sum}. 

Now put $\Theta=(\log Y)/\sL$. Noting that
$$
X^cN^{-\sigma}\ll hd^{-1}x^c\ll
X^cN^\sigma\qquad(x\sim X)
$$
and $X\asymp N^{1-\Theta}$, in the Type~I case we  only need to show that
\begin{equation}
\label{eq:panther}
\sum_{z\sim N^\Theta}\e(z^cN^\Delta)\ll N^{\Theta-\sigma-\eps}
\end{equation}
holds uniformly for
\begin{equation}
\label{eq:panther2}
1/2-\beta\lealmost\Theta\le 1
\end{equation}
and
\begin{equation}
\label{eq:panther3}
(1-\Theta)c-\sigma\lealmost\Delta\lealmost (1-\Theta)c+\sigma.
\end{equation}

Suppose first that $\Theta,\Delta$ are such that~\eqref{eq:panther2} 
and~\eqref{eq:panther3} hold,
and fix $\eps>0$.
Define $k$ by~\eqref{eq:k_defn} and $\varrho$
by~\eqref{eq:rho_defn}.  Note that $\varrho=f(k)$,
where
$$
f(t)=\frac{t-2-\eps}{t(t+1)(2t-1)}.
$$
Since $f$ is decreasing on $[3,\infty)$, and noting that
the bounds
$$
3\le k\le c+\Delta/\Theta+1\lealmost (c+\sigma)/\Theta+1
$$
hold in view of~\eqref{eq:panther2} and~\eqref{eq:panther3}, it follows that
$$
\Theta\varrho\gealmost f_1(\Theta),
$$
where
$$
f_1(t)=\frac{c_1t^3-(1+\eps)t^4}{(c_1+t)(c_1+2t)(2c_1+t)}
$$
with $c_1=c+\sigma$.
Since $c\ge 1.6$ and $\eps\le 0.01$ (say), the function $f_1$
is increasing on $[0,1]$; consequently,
as $\Theta\gealmost 1/2-\beta$ we have
$$
\Theta\varrho\gealmost f_1(1/2-\beta)=
\frac{c_1(\tfrac12-\beta)^3-(1+\eps)(\tfrac12-\beta)^4}
{(c_1+\tfrac12-\beta)(c_1+1-2\beta)(2c_1+\tfrac12-\beta)}. 
$$
In view of~\eqref{eq:c1ineq} we can choose $\eps>0$ sufficiently small,
depending only on $c$, such that
$$
\Theta\varrho\gealmost\sigma+\eps.
$$
Then, using the equation~\eqref{eq:S(N)bd} of Lemma~\ref{lem:gerbils},
we derive the required bound~\eqref{eq:panther} for the Type~I sums~\eqref{eq:S1defx}.

Next, suppose that $\Theta,\Delta$ are such that~\eqref{eq:solar2} 
and~\eqref{eq:solar3} hold,
and let $\eps>0$ be chosen as above.
We again define $k$ by~\eqref{eq:k_defn}
and put $\varrho=f(k)$.  Since $f$ is decreasing on $[3,\infty)$,
and noting that the bounds
$$
3\le k\le c+\Delta/\Theta+1\lealmost (c-1+3\sigma+2\eps)/\Theta+2
$$
hold in view of~\eqref{eq:solar2} and~\eqref{eq:solar3}, it follows that
$$
\Theta\varrho\gealmost f_2(\Theta),
$$
where
$$
f_2(t)=\frac{(c_2+2\eps)t^3-(1+\eps)t^4}{(c_2+2t+2\eps)(c_2+3t+2\eps)(2c_2+3t+4\eps)}
$$
with $c_2=c-1+3\sigma$.
Since $c\ge \frac{11}{5}$ and $\eps\le 0.01$,
one verifies that $f_2$ attains a unique maximum on $[0,1]$;
therefore, as $2/3\lealmost\Theta\lealmost 1-2\beta$ we have either
$$
\Theta\varrho\gealmost
f_2(2/3)=\frac{\tfrac8{27}(c_2+2\eps)
-\tfrac{16}{81}(1+\eps)}{(c_2+\tfrac43+2\eps)(c_2+2+2\eps)(2c_2+2+4\eps)}
$$
or else
\begin{align*}
\Theta\varrho&\gealmost f_2(1-2\beta)\\
&=\frac{(c_2+2\eps)(1-2\beta)^3-(1+\eps)(1-2\beta)^4}
{(c_2+2-4\beta+2\eps)(c_2+3-6\beta+2\eps)(2c_2+3-6\beta+4\eps)}.
\end{align*}
In view of the inequalities~\eqref{eq:c2ineqA}
and~\eqref{eq:c2ineqB},
we can take $\eps>0$ sufficiently small
to guarantee that
$$
\Theta\varrho\gealmost 2\sigma+3\eps.
$$
Using the equation~\eqref{eq:S(N)bd} of Lemma~\ref{lem:gerbils} once again,
we derive the required bound~\eqref{eq:solar} for the Type~II sums~\eqref{eq:S2defx}.

\section*{Acknowledgements}

We thank Roger Baker for his generous help and valuable 
advice, and for sharing his ideas.  In particular, our proofs of the crucial
Lemmas~\ref{lem:GenMultSum} and~\ref{lem:gerbils} were originally sketched by Roger Baker. 
We are also grateful to Xiaodong Cao and Wenguang Zhai for informing 
us about their paper~\cite{CaoZhai}. 

During the preparation of this paper,
I.~E.~Shparlinski was supported in part by ARC grants DP130100237 and~DP140100118.

\end{document}